\documentclass[12pt,reqno]{amsart}
\usepackage{url}
\usepackage{color}
\usepackage{amssymb}
\let\oldlabel=\label
\def\prellabel{\marginparsep=1em\marginparwidth=44pt
  \def\label##1{\oldlabel{##1}\ifmmode\else\ifinner\else
         \marginpar{{\footnotesize\ \\ \tt
                    ##1}}\fi\fi}}
\def\NZQ{\mathbb}               
\def\NN{{\NZQ N }}
\def\QQ{{\NZQ Q}}

\def\VV{V}

\def\opn#1#2{\def#1{\operatorname{#2}}} 
\opn\chara{char}
\opn\rank{rank}
\opn\hilb{Hilb}
\opn\gr{gr}
\opn\Rees{{\mathcal R}}
\usepackage[latin1]{inputenc}
\usepackage{mathptmx}
\newtheorem{theorem}{Theorem}[section]
\newtheorem{lemma}[theorem]{Lemma}
\newtheorem{corollary}[theorem]{Corollary}
\newtheorem{proposition}[theorem]{Proposition}
\newtheorem{assumption}[theorem]{Assumption}

\theoremstyle{definition}
\newtheorem{remark}[theorem]{Remark}
\newtheorem{definition}[theorem]{Definition}
\newtheorem{example}[theorem]{Example}
\newtheorem{remark/example}[theorem]{Remark/Example}

\newtheorem*{theorem*}{Theorem}

\let\epsilon=\varepsilon
\let\phi=\varphi
\let\kappa=\varkappa

\opn\ini{in}
\opn\KRS{KRS}
\opn\krs{krs}
\opn\Krs{Krs}
\opn\DEL{DEL}
\opn\diag{diag}
\opn\Ker{Ker}
\opn\Image{Im}
\opn\DD{{\mathcal D}}
\opn\SS{{\mathcal S}}
\opn\MM{{\mathcal M}}
\opn\GL{GL}
\opn{\hht}{ht}
\opn\Cl{Cl}
\opn\cl{cl}
\opn\height{height}
\opn\reg{reg}
\opn\Reg{Reg}
\opn\pd{pd}
\opn\supp{supp}
\opn\depth{depth}
\opn{\HS}{HS}
\opn{\HF}{HF}
\opn{\HP}{HP}

\opn{\Ass}{Ass}
\opn{\Min}{Min}

\opn{\codim}{codim} 
\opn{\image}{Im}

\def\C{\mathcal C}

\def\mm{{\mathfrak m}}

\def\addots{\mathinner{\mkern1mu\raise1pt\hbox{.}\mkern2mu\raise4pt\hbox{.}
        \mkern2mu\raise7pt\vbox{\kern7pt\hbox{.}}\mkern1mu}}

\numberwithin{equation}{section}

\makeatletter
\def\author@andify{%
  \nxandlist {\unskip ,\hspace{3em}\ignorespaces}%
  {\unskip {}\hspace{3em}\ignorespaces}
  {\unskip ,\hspace{3em}\ignorespaces}%
}
\makeatother

\begin{document}

\title{A  complex associated to a family of ideals and applications to subspace arrangements}

\author{Aldo Conca}
\address{Dipartimento di Matematica,  Universit\`a di Genova \\Via Dodecaneso 35, 16146 Genova, Italy}
\email{aldo.conca@unige.it}

\author{Manolis C. Tsakiris}
\address{Academy of Mathematics and Systems Science \\Chinese Academy of Sciences, 100190, Beijing, China}
\email{manolis@amss.ac.cn}
 
 \thanks{This research was conducted while the authors were visiting SLMath/UC-Berkeley  within the framework of the Commutative Algebra semester in 2024. The authors gratefully acknowledge the support of the host institution.}

\thanks{M.C. Tsakiris is supported by the National Key R$\&$D Program of China (2023YFA1009402).  A. Conca  was partially supported by MIUR Excellence Department Project awarded to the Dept.~of Mathematics, Univ.~of Genova, CUP D33C23001110001,  by PRIN 2020355B8Y  ``Squarefree Gr\"obner degenerations, special varieties and related topics"  and by GNSAGA-INdAM}
 


\maketitle

\begin{center}
\textit{Dedicated to the memory of Winfried Bruns.}
\end{center}

\begin{abstract}
 We reprove and generalize a result of Derksen describing the Hilbert polynomial of the vanishing ideal of a transversal subspace arrangement in projective space. When the subspaces have all codimension $2$, we describe explicitly the local cohomology modules of the corresponding defining ideal. This is achieved via a complex naturally associated to the subspace arrangement, whose basic properties are also investigated.
 \end{abstract} 

\section*{Introduction} 
We introduce a complex $\C$ naturally associated to a family of ideals $I_1,\dots,I_n$ in a commutative ring $S$  whose $0$-th cohomology is the intersection ideal $\cap_{i=1}^n I_i$.  This complex  (indeed  a minor variation of it) appears  also in the recent preprint  \cite{BCH} of Banerjee,  Chardin and  Holanda,  but for different purposes.  
We study elementary conditions under which the higher cohomology modules of  $\C$ do vanish. 

 In the case of a family of ideals associated to an arrangement $\VV=V_1,\dots, V_n$ of subspaces of linear forms of $\mathbb{K}[x_1,\dots,x_r]$, we use the  complex $\C$ to prove and generalize a result of Derksen \cite{Derksen} describing the Hilbert polynomial of the ideal $I_\VV = \cap_{i \in [n]} (V_i)$. This is the vanishing ideal of the dual subspace arrangement $\mathbb{V}(\VV)= \cup_{i \in [n]} \mathbb{V}(V_i)$ in projective space $\mathbb{P}^{r-1}$.
Furthermore we show that the positive cohomology modules of $\C$ have finite length. When all the $V_i$ have the same dimension, the complex can be used to give a quite explicit description of the local cohomology modules of $I_\VV$. We do this in details when $\dim V_i=2$ for all $i$. It turns out that, apart from a small exception, the $j$-th local cohomology module of $I_\VV$ is either equal to a cohomology module of $\C$ (when $j$ is small) or to equal to the direct sum of the local cohomology modules of polynomial rings (when $j$ is large). 
As a corollary we have that the top non-vanishing degrees of the cohomology modules of $\C$ determine the regularity of $I_\VV$. 

\section{A complex associated to a family of ideals}

 Let $S$ be a commutative ring with $1$. For a pair of ideals $I_1$ and $I_2$ of $S$ we have a well known short exact sequence 
 \begin{equation}
 \label{MV}
 0\to S/(I_1\cap I_2) \to S/I_1\oplus S/I_2 \to S/(I_1+I_2)\to 0
 \end{equation} 
 
 We extend this construction to a family of ideals $I_1,I_2,\dots I_n$ as follows. 
 
 Let $F$ be a free $S$-module of rank $n\in \NN$ with given basis $e_1,\dots,e_n$.   The exterior algebra 
 $$E=   \wedge  F= \oplus_{i=0}^n \wedge^i F=\oplus_{i=0}^n E_i$$ equipped with the differential induced by the multiplication with $L=\sum_{i=1}^n e_i\in E_1$ is an exact complex: 
$$(E,L):  \quad 0\to E_0=S \to E_1=F \to E_2  \to \dots \to E_n  \to  0.$$ 
As an $S$-module $E_i$ is free with basis the set of elements of the form $e_A=e_{a_1}\wedge \dots \wedge e_{a_i}$, where $A=\{a_1,a_2,\dots, a_i\}\subseteq [n]$ and $a_1<\dots<a_i$.

 Given ideals $I_1,I_2,\dots, I_n$ of $S$ we set  $I_A=\sum_{i\in A} I_i$ for every $A\subseteq [n]$. If $A=\{i,j\}$ we will use $I_{i,j}$ for $I_{\{ij\}}$. We consider  in the algebra $E$ the homogeneous (and hence two sided) ideal ${\mathcal I}$ generated by $I_1e_1, I_2e_2, \dots,  I_ne_n$, that is: 
 $${\mathcal I}=\oplus_{A}  I_Ae_A.$$  
As  ${\mathcal I}$  is an ideal of $E$ it is also a subcomplex of $(E,L)$.  It follows that $E/\!\!\mathcal{I}$ is a (ring and a) complex of $S$-modules with differential induced by the multiplication with $L$ on  it.  We denote this complex  by $\C(I_1,\dots,I_n)$, that is
$$\C(I_1,\dots,I_n)=(E/\!\!\mathcal{I},L)$$
and denote by $H^i(I_1,\dots, I_n)$ its $i$-th cohomology. 

\begin{example} 
\label{exn=2}
For $n=2$ we obtain the complex: 
  $$\C(I_1,I_2) :\quad  0\to S\to (S/I_1)e_1\bigoplus (S/I_2)e_2\to (S/I_1+I_2)e_{12} \to 0$$  
  with cohomology $H^0(I_1,I_2)=I_1\cap I_2$ and $H^i(I_1,I_2)=0$ for $i=1,2$. Hence $\C(I_1,I_2)$  induces the short exact sequence (\ref{MV}). 
  \end{example} 
  
  \begin{example} 
 For $n=3$ we have: 
  $$\C(I_1,I_2,I_3): \quad 0\to S\to  \bigoplus_{i}  (S/I_i)e_i \to \bigoplus_{i<j}  (S/I_i+I_j)e_{ij}  \to   (S/I_1+I_2+I_3) e_{123} \to 0$$  
  with   $H^0(I_1,I_2,I_3)=I_1\cap I_2\cap I_3$. We will see in \ref{T3<>0} that $H^i(I_1,I_2,I_3)=0$ for $i=2,3$  while $H^1(I_1,I_2,I_3)$ can be non-zero.  
  \end{example} 
  
  \begin{remark} 
  \label{easy}
  We list some   properties of the complex $\C(I_1,\dots,I_n)$ whose verifications are straightforward.  
  \begin{itemize} 
  \item[(1)] If $I_n=S$ then $\C(I_1,\dots,I_n)=\C(I_1,\dots,I_{n-1})$. 
  \item[(2)] For a multiplicative set $T$ of $S$ one has $\C(I_1,\dots,I_n)_T=\C(I'_1,\dots,I'_n)$ where $I_i'=I_iS_T$. 
  \item[(3)] The complex  $\C(I_1,\dots,I_n)$ can be obtained also as the tensor product of complexes
  $$\bigotimes_{i=1}^n  (S\to S/I_i)$$
  with $S$ in homological position $0$ and $S/I_i$ in homological position $1$. 
  \item [(4)] $\C(I_1,\dots,I_n)\otimes S/J=\C(I'_1,\dots,I'_n)$ with $I_i'=I_i+J/J$, for any ideal $J$. 
  \end{itemize}
  \end{remark} 
  
  The $0$-th cohomology of the complex $\C(I_1,\dots,I_n)$ is easily identified: 
  
 $$H^0(I_1,\dots,I_n)=\cap_{i=1}^n I_i$$
 
But what can we say about  the higher cohomology? In particular, for which collections of ideals $I_1,\dots, I_n$ does it vanish? To answer these questions we make explicit the inductive structure of the complex  $\C(I_1,\dots,I_n)$.    
 
 \begin{lemma}
 \label{seqC}
  Given ideals $I_1,\dots, I_n$ of $S$, consider the ideals $J_1,\dots, J_{n-1}$ of $S/I_n$ defined as  $J_i=(I_i+I_n)/I_n$.  We have a short exact sequence of complexes: 
 $$0\to \C(J_1,\dots, J_{n-1})(-1) \to \C(I_1,\dots,I_n)\to  \C(I_1,\dots,I_{n-1})\to 0.$$
 \end{lemma} 
 
 \begin{proof} In the ring $\C(I_1,\dots,I_n)=E/\!\!\mathcal{I}$ we consider the ideal $H$ generated by $e_n$.  
 By construction $H$ is a subcomplex of $\C(I_1,\dots,I_n)$ so that we have a short exact sequence of complexes: 
 $$0\to H \to \C(I_1,\dots,I_n) \to \C(I_1,\dots,I_n)/H\to 0.$$
 Now we  note that $H$ gets identified with $\C(J_1,\dots, J_{n-1})(-1)$ because the multiplication by $L=\sum_{i=1}^n e_i$ on $H$ equals the multiplication by $\sum_{i=1}^{n-1} e_i$. Finally  $\C(I_1,\dots,I_n)/H$ gets identified with $\C(I_1,\dots,I_{n-1})$ because the multiplication by  $L=\sum_{i=1}^n e_i$ modulo $e_n$ equals the multiplication by  $\sum_{i=1}^{n-1} e_i$. 
 \end{proof} 
  
As a corollary we have: 
 
\begin{corollary} 
\label{longex}
With the notation of Lemma \ref{seqC} we set 
$$H^i_n=H^i(I_1,\dots,I_n),  \quad H^i_{n-1}=H^i(I_1,\dots,I_{n-1}) \mbox{ and }H^i(J)=H^i(J_1,\dots, J_{n-1}).$$ Then we have a long exact sequence 
 $$0\to T_n \to H^1_n \to H^1_{n-1}\to H^1(J)\to H^2_n \to H^2_{n-1}\to H^2(J)\to H^3_n \to  \dots, $$ 
where $$T_n=\frac{ \bigcap_{i=1}^{n-1}  I_{i,n}   }{  (\bigcap_{i=1}^{n-1}  I_i)+I_n}.$$
 \end{corollary} 
 
 \begin{proof} The short exact sequence of complexes  \ref{seqC} gives a long exact sequence: 
$$0\to H^0_n \to  H^0_{n-1}  \stackrel{f} \to H^0(J) \to H^1_n  \to  H^1_{n-1} \to H^1(J) \to \cdots $$
Since   $H^0_{n-1}=\cap_{i=1}^{n-1} I_i$,  $H^0(J)=\cap_{i=1}^{n-1} I_{i,n}/I_n$   and since the connecting homomorphism $f$ is induced by the inclusion 
$$\cap_{i=1}^{n-1} I_i \subseteq  \cap_{i=1}^{n-1} I_{i,n} $$ 
we obtain the desired assertion. 
\end{proof} 

\begin{corollary} 
For all $n>1$ we have $H^i(I_1,\dots, I_n)=0$ for $i=n-1,n$. 
\end{corollary} 
\begin{proof} The assertion follows immediately from \ref{longex}
by induction on $n$, the base case $n=2$ being \ref{exn=2}. 
\end{proof}

\begin{example}
\label{T3<>0}  
 When $n=2$ one has $H^i(I_1,I_2)=0$ for $i=1,2$. 
 When $n=3$ one has  $H^i(I_1,I_2,I_3)=0$ for $i=2,3$ and 
$$H^1(I_1,I_2,I_3)\simeq T_3=\frac{I_{1,3}\cap I_{2,3} }{I_1\cap I_2 +I_3}.$$
 In particular with $S=K[x,y]$, $I_1=(x), I_2=(x+y), I_3=(y)$ one has 
$H^1(I_1,I_2,I_3)=(x,y)/(x^2,y)\neq 0$. 
\end{example}  
 
The discussion above indicates that the vanishing of $H^1(I_1,\dots, I_n)$ implies the vanishing of $T_n$ which is, in turn, related to a distributive property (of sum over intersection) of the family of ideals $I_1,\dots,I_n$.  This observation leads to the following:  
 
 \begin{theorem} 
 \label{vanishing} 
 Suppose that for every $1<j\in [n]$ and every $B\subseteq  \{j+1,\dots, n\}$  
 one has: 
 $$\cap_{i=1}^j (I_i+I_B)=(\cap_{i=1}^j I_i)+I_B.$$
Then $H^i(I_1,\dots,I_n)=0$ for every $i>0$. 
\end{theorem}

\begin{proof} By induction on $n$. The case $n=2$ is trivial. For $n>2$ we use the notations of Corollary \ref{longex}. By induction we have  that $H_{n-1}^i=0$ for all $i>0$. Now we claim   that the family $J_1,\dots, J_{n-1}$ satisfies the assumption of Theorem \ref{vanishing}  as well. To this end we have to prove that  $1<j\in [n-1]$ and every $B\subset \{j+1,\dots, n-1\}$  one has 
$\cap_{i=1}^j (J_i+J_B)=(\cap_{i=1}^j J_i)+J_B$. Equivalently, one has to check  that 
$$\cap_{i=1}^j \big(I_i+I_{B\cup \{n\}}\big)=\big(\cap_{i=1}^j (I_i+I_n)\big)+I_{B\cup \{n\}}.$$
Then inclusion $\supseteq $ is obvious. For the reverse inclusion one observes that, by assumption, 
$$\cap_{i=1}^j \big(I_i+I_{B\cup \{n\}}\big)=\big(\cap_{i=1}^j I_i\big)+I_{B\cup \{n\}}$$
and the left hand side is clearly contained in $(\cap_{i=1}^j I_i+I_n)+I_{B\cup \{n\}}$. Hence by induction we have also that $H^i(J)=0$ for all $i>0$. Finally one observes that   the assumption with $j=n-1$ and $B=\{n\}$ actually gives  that $T_n=0$. Now it follows from the long exact sequence of Corollary \ref{longex} that $H^i_n=0$ for all $i>0$. 
\end{proof} 

\begin{definition} 
Let $S$ be a $K$-algebra and $B$ be a $K$-basis of $S$.  We say that an ideal $I$ of $S$ is $B$-compatible if $B\cap I$ is a $K$-basis of $I$. 
\end{definition} 

\begin{remark} The set of $B$-compatible ideals is closed under sum and intersection. Furthermore the sum is distributive over the intersection for $B$-compatible ideals $I,J,K$: 
$$(I+J)\cap K=(I\cap K)+(J\cap K).$$ 
In particular, for $B$-compatible ideals $I_1,I_2,\dots, I_n,J$ of $S$:
$$\cap_{i=1}^n(I_i+J)=\left(\cap_{i=1}^n I_i \right)+J$$
\end{remark} 

 \begin{corollary} 
 \label{vanishing2} Let $S$ be a $K$-algebra and $B$ be a $K$-basis of $S$.  
Let $I_1,\dots, I_n$ be ideals of $S$ that are $B$-compatible.  Then $H^i(I_1,\dots, I_n)=0$ for every $i>0$. 
\end{corollary}  

\begin{proof}  Under the current assumptions we have that   $\cap_{i=1}^j (I_i+I_B)=(\cap_{i=1}^j I_i)+I_B$ holds  and the conclusion follows from Theorem \ref{vanishing}.  
 \end{proof} 

Monomial ideals in polynomial rings or, more generally, in toric rings are compatible with the monomial basis. Hence we have: 
\begin{theorem} \label{thm:monomial}
\label{monomialcase} 
Let $S$ be a polynomial ring over a field $K$ or a toric subring of it.  
Let $I_1,\dots, I_n$ be monomial ideals in $S$. Then $H^i(I_1,\dots, I_n)=0$ for every $i>0$.  
\end{theorem}

\section{The Hilbert polynomial of linearly general subspace arrangements} 
 
 Let $S=K[x_1,\dots,x_r]$ and $\VV=V_1,\dots, V_n$   be a subspace arrangement of $S_1$. Denote by  $I_i$ the ideal of $S$ generated by $V_i$ and by $I_{\VV}$ the intersection $I_1\cap I_2\cap \dots \cap I_n$.   For a subset $A\subseteq  [n]$ we set $V_A=\sum_{i\in A} V_i$, $d_A=\dim V_A$ and $c_A= \sum_{i\in A} \dim_K V_i$. 
 We set 
  $$P_{\VV}(x)=\sum_A (-1)^{|A|} { x+r-d_A -1 \choose r-d_A -1}$$ 
  where the sum is over all the $A\subseteq [n]$ (including the empty set) with $d_A<r$. 
  
Recall that the arrangement $\VV$ is said to be linearly general if for every $A\subseteq  [n]$ such that $d_A<r$ one has  $d_A=c_A$. This is equivalent to saying the dual arrangement $\mathbb{V}(\VV)$ of $\mathbb{P}^{r-1}$ is transversal in the sense of \cite{Derksen}.
 
 Derksen proved in \cite{Derksen} that $P_{\VV}(x)$ is the Hilbert polynomial of $I_{\VV}$ provided $\VV$ is linearly general. We will reprove and generalize this result. To this end we generalize the definition of linearly general  as follows. 
 
\begin{definition} Let $c\in \NN$. We say that $\VV$ is linearly general in codimension $c$ if for every $A\subseteq [n]$ such that $d_A<r-c$ one has  $d_A=c_A$. 
\end{definition} 
 
 We have:
 
\begin{proposition}
\label{smalldim}  Assume  $\VV$ is linearly general in codimension $c$. Then $H^i(I_1,\dots,I_n)$ has Krull dimension $\leq c$ for all $i>0$. 
\end{proposition} 
  
  \begin{proof} Let $P$ be a prime ideal of $S$ of codimension $<r-c$. We have to prove that $H^i(I_1,\dots,I_n)_P=0$ for $i>0$. Let 
  $$A=\{ j\in [n] : I_j\subseteq P\}.$$ Then by Remark \ref{easy} $\C(I_1,\dots,I_n)_P=\C(I_j : j\in A)_P$ so that $H^i(I_1,\dots,I_n)_P=H^i(I_j : j\in A)_P$. Now $V_A \subset P$ by construction so that $d_A \leq \codim P<r-c$. Hence, by assumption, $d_A=c_A$. But then we can choose a basis of  $S_1$ such that each $V_j$ with $j\in A$ is generated by a subset of the basis. In other words, up to a change of coordinates we may assume that the ideals $I_j$ with $j\in A$ are generated by (disjoint sets of) variables. In particular, they are monomial ideals and hence,  by Theorem \ref{monomialcase}, $H^i(I_j : j\in A)=0$ for all $i>0$. And this implies that $H^i(I_j : j\in A)_P=0$ for all $i>0$. 
  \end{proof}

As a corollary we have: 
  
  \begin{theorem} 
  Assume  $\VV$ is  linearly general in codimension $c$.
  Then the Hilbert polynomial of $I_{\VV}$ coincides with $P_{\VV}(x)$   in all terms of degree $\geq c$. 
  \end{theorem} 
  \begin{proof} Set $P_M(x)$ the Hilbert polynomial of a graded $S$-module $M$. Consider the complex 
 $$\C(I_1,\dots,I_n): 0\to C_0\to C_1 \to C_2 \to \dots \to  C_n \to 0,$$ 
 and set $H^i=H^i(I_1, \dots, I_n)$ so that $I_{\VV}=H^0$. 
 Then  
 $$\sum_{i=0}^n  (-1)^iP_{C_i}(x)=\sum_{i=0}^n  (-1)^iP_{H^i}(x).$$
 By \ref{smalldim}, $\dim H^i\leq c$ for all $i>0$ and so $\deg P_{H^i}(x)<c$. 
  Hence 
 $$\sum_{i=0}^n  (-1)^iP_{C^i}(x)=P_{H^0}(x) \mbox{  in all terms of degree } \geq c.$$ 
 On the other hand 
 $$\sum_{i=0}^n  (-1)^iP_{C_i}(x)=\sum_{A\subseteq [n]}  (-1)^{|A|} { x+r-d_A -1 \choose r-d_A -1}.$$
 The latter sum can be restricted to the summands with $d_A<r$ because the contribution of the summands with $d_A =r$ vanishes, that is, 
 $$\sum_{i=0}^n  (-1)^iP_{C_i}(x)=P_{\VV}(x).$$
 
 Therefore we may conclude that 
 $$P_{\VV}(x)=P_{H^0}(x) \mbox{  in all terms of degree } \geq c$$ 
 and this completes the proof. 
 \end{proof}

  \begin{example} 
  In $S=\QQ[a,b,c,d,e]$ we consider $I_1=(a,b), I_2=(c,d), I_3=(a + c, e), I_4=(b+e, c+e)$. 
  This subspace arrangement is linearly general in codimension $1$ but not linearly general because 
  $I_1+I_3+I_4=(a,b,c,e)$. 
  One has  
  $$P_{\VV}(x)= \frac{1}{24}x^4 + \frac{5}{12}x^3 - \frac{13}{24}x^2 - \frac{47}{12}x + 2,$$
  while the Hilbert polynomial of $I_{\VV}$ is 
  $$P_{I_{\VV}}(x)=\frac{1}{24}x^4 + \frac{5}{12}x^3 - \frac{13}{24}x^2 - \frac{47}{12}x  + 4.$$
Also $P_{H^1}(x)=2$ and $H^i=0$ for $i=2,3,4$. 
  All computations done with Macaulay 2 \cite{M2}. 
       \end{example} 
    
Our next goal is to identify the local cohomology modules of the ideal $I_{\VV}$ when $\VV$ is linearly general and $\dim V_i=2$ for all $i$. The next section contains the results that would be useful to reach this goal. 
   
 \section{Complexes with separated cohomology} 
 Let $S=K[x_1,\dots, x_r]$. Let $s\in \NN$ with $s>0$. 
 Suppose we have a complex of finitely generated graded $S$-modules 
$$\MM: 0\to M_0 \stackrel{f_0}\longrightarrow  M_1\stackrel{f_1}\longrightarrow M_2 \stackrel{f_2}\longrightarrow  \dots  \to M_{s} \to 0.$$ 
 Let $H^i(\MM)$ be the $i$-cohomology of $\MM$ at $M_i$ and set 
$$N=H^0(\MM).$$
Splitting the complex into short exact sequences, a simple inductive argument gives the following inequality: 

$$\reg(N)\leq \max
 \left\{ 
 \begin{array}{lll}
  \reg(H^i(\MM))+i+1 &:&  i=1,\dots, s, \\
  \reg(M_i)+i  &:& i=0,\dots, s.
 \end{array} 
 \right.
 $$

We will show that under some assumptions the inequality above becomes an equality. 

Assume that: 

\begin{assumption}
\label{sepcoho} 
\begin{itemize}  
\item[(1)] $H^i(\MM)$ has Krull dimension $0$  for all $i=1,\dots, s$. 
\item[(2)] (even case)   $M_i$ is Cohen-Macaulay of dimension  $2(s-i+1)$ for all $i=0,\dots, s$ \\ or  
\item[(3)] (odd case) $M_i$ is Cohen-Macaulay of dimension  $2(s-i+1)-1$ for all $i=0,\dots, s$.   
\end{itemize} 
\end{assumption} 

We will see that, under these assumptions, the $j$-th local cohomology module of $N$  will be either one of the cohomology modules of $\MM$ (when $j$ is small) or the local cohomology module of one of the $M_i$'s (when $j$ is large) with one small exception in the odd case.  The precise statements are below. 
The proofs are easy and follow by induction on $s$ and by splitting the complex into short exact sequences. We will  present them  in  details in the even case and then we will highlight the differences in the odd case. 

\begin{theorem}\label{thm:even}
\label{evencase} 
In the even case one has: 
$$H_\mm^{i}(N)=\left\{ 
\begin{array}{ll}
H^{i-1}(\MM) &  \mbox{ for } i=2,\dots, s+1 \\ 
H_\mm^{2i-2s-2}(M_{2s-i+2}) & \mbox{ for } i=s+2,\dots, 2s+2\\
 0 &  \mbox{  for  }  i<2 \mbox{ or } i>2s+2 
 \end{array}
 \right.
 $$
 \end{theorem} 
\begin{proof} 
Let us start to analyze the case  $s=1$. We have a complex
$$0\to M_0 \stackrel{f_0} \longrightarrow M_1 \to 0$$ 
with $M_0$ and $M_1$ Cohen-Macaulay modules of dimension $4$ and $2$ and the $1$-st cohomology is $0$-dimensional. We break the complex into two  exact sequences: 
$$(1) \quad 0\to N \longrightarrow M_0 \longrightarrow \image f_0 \to 0  \qquad 
(2) \quad 0\to \image f_0  \longrightarrow M_1 \longrightarrow H^1(\MM) \to 0 $$
Now $\image f_0\neq 0$ otherwise $M_1\simeq H^1(\MM)$ which is a contradiction and $\image f_0$  has dimension $\leq 2$.  Then 
$\depth \image f_0>0$ being $\image f_0\subseteq  M_1$. Therefore the long exact sequence of local cohomology modules associated to  (1) gives $H_\mm^0(N)=H_\mm^1(N)=0$ and 
$$
H_\mm^2(N)=H_\mm^1(\image f_0),\quad  H_\mm^3(N)=H_\mm^2(\image f_0), \quad
H_\mm^4(N)=H_\mm^4(M_0).$$

Now the long exact sequence of local cohomology modules associated to  (2) gives: 
$$H_\mm^1(\image f_0)=H^1(\MM), \quad H_\mm^2(\image f_0)=H_\mm^2(M_1).$$

Summing up  we have: 
$$H_\mm^i(N)=\left\{ 
\begin{array}{ll}
H^1(\MM) & \mbox{ for } i=2\\
 H^2_\mm(M_1) & \mbox{ for } i=3\\
 H^4_\mm(M_0) & \mbox{ for } i=4\\
 0 &  \mbox{ for  } i=0,1 \mbox{ and } i>4
 \end{array}
 \right.
 $$
 
Now we assume that $s>1$ and argue by induction on $s$. 

We break the complex into two  exact sequences
$$(3) \quad 0\to N \longrightarrow M_0 \longrightarrow \image f_0 \to 0  \qquad 
(4) \quad 0\to \image f_0  \longrightarrow \ker f_1  \longrightarrow H^1(\MM) \to 0 $$
and a complex: 
$$(5) \MM': 0\to M_1 \stackrel{f_1}\longrightarrow  M_2\stackrel{f_2}\longrightarrow M_3 \stackrel{f_3}\longrightarrow  \dots  \to M_{s} \to 0.$$ 
and clearly $\MM'$ satisfies the assumptions \ref{sepcoho}.   By induction on $s$ we know how to express the local cohomology modules of $\ker f_1$ in terms of $H^i(\MM')=H^{i+1}(\MM)$ with $i=1,\dots, s-1$ and $H_\mm^{2(s-j+1)}(M_j)$ with $j=1,\dots,s$. In particular, $\depth \ker f_1>1$. Now, the long exact sequence of local cohomology modules associated to (3) gives $H_\mm^0(N)=H_\mm^1(N)=0$  and 
$$ H_\mm^i(N)=H_\mm^{i-1}(\image f_0) \mbox{ for } i=2,\dots, 2s+2-1, \quad
H_\mm^{2s+2}(N)=H_\mm^{2s+2}(M_0),$$
where the last equality is because $\image f_0$ is a submodule of $M_1$ whose dimension is $2s$. Taking into consideration that $\depth \Ker f_1>1$, the long exact sequence of local cohomology modules associated to (4) gives: 
$$H^1_\mm(\image f_0)=H^1(\MM), \quad H^i_\mm(\image f_0)=H^i_\mm(\Ker f_1) \mbox{ for } i>0.$$
Collecting all the data and paying attention at the indices one obtains the desired expression for $H^i_\mm(N)$. \end{proof}

\begin{theorem}\label{thm:odd}
\label{oddcase} 
In the odd  case one has: 
$$H_\mm^{i}(N)=\left\{ 
\begin{array}{ll}
H^{i-1}(\MM) &  \mbox{ for } i=2,\dots, s \\ 
H_\mm^{2i-2s-1}(M_{2s-i+1}) & \mbox{ for } i=s+2,\dots, 2s+1\\
 0 &  \mbox{  for  }  i<2 \mbox{ or } i>2s+1 
 \end{array}
 \right.
 $$
 and $H^{s+1}_\mm(N)$ sits in the following short exact sequence: 
$$0\to H^s(\MM) \to H^{s+1}_\mm(N) \to H^1_\mm(M_s)\to 0.$$  
 \end{theorem} 
 \begin{proof} 
Let us start to analyze the case  $s=1$. We have a complex
$$0\to M_0 \stackrel{f_0} \longrightarrow M_1 \to 0$$ 
with $M_0$ and $M_1$ Cohen-Macaulay modules of dimension $3$ and $1$ and the $1$-st cohomology is $0$-dimensional. We break the complex into two  exact sequences: 
$$(1) \quad 0\to N \longrightarrow M_0 \longrightarrow \image f_0 \to 0  \qquad  
(2) \quad 0\to \image f_0  \longrightarrow M_1 \longrightarrow H^1(\MM) \to 0 $$
Now $\image f_0\neq 0$ otherwise $M_1\simeq H^1(\MM)$ which is a contradiction and $\image f_0$  has dimension $\leq 1$.  Then 
$\depth \image f_0>0$ being $\image f_0\subseteq  M_1$. Therefore the long exact sequence of local cohomology modules associated to  (1) gives: 
$$H_\mm^0(N)=H_\mm^1(N)=0,  \quad
H_\mm^2(N)=H_\mm^1(\image f_0),\quad  H_\mm^3(N)=H_\mm^3(M_0).$$
Now the long exact sequence of local cohomology modules associated to  (2) gives a short exact sequence
 $$0\to H^1(\MM) \to H^1_\mm(\image f_0) \to H^1_\mm(M_1)\to 0.$$  
Summing up  we have: 
$$H_\mm^i(N)=\left\{ 
\begin{array}{ll}
 H^3_\mm(M_0) & \mbox{ for } i=3\\
 0 &  \mbox{ for  } i=0,1 \mbox{ and } i>3
 \end{array}
 \right.
 $$
 and $H^2_\mm(N)$ sits in the short exact sequence: 
 $$0\to H^1(\MM) \to H^2_\mm(N)\to H^1_\mm(M_1)\to 0$$
 This proves the assertion for $s=1$. The rest of the proof follows by using the argument already employed in the even case. 
\end{proof}

 It follows that,  both in the even and odd case,  one has: 
 
 \begin{corollary} 
 \label{regul}
 $$\reg(N)= \max
 \left\{ 
 \begin{array}{lll}
  \reg(H^i(\MM))+i+1 &:&  i=1,\dots, s, \\
  \reg(M_i)+i  &:& i=0,\dots, s.
 \end{array} 
 \right.
 $$
\end{corollary} 

\section{Local cohomology for subspace arrangements of dimension $2$} 
In this section we will use the results obtained so far to describe the local cohomology modules of  $I_{\VV}$ for a linearly general subspace arrangement $\VV$ of dimension $2$ (recall notations and definitions from section 2).
 
 \begin{assumption} \label{ass:LinearlyGeneralDim=2}
 In this section we assume that  $\dim V_i=2$ for every $i$ and that the arrangement $\VV$ is linearly  general. 
 \end{assumption} 
 
 Consider the complex 
 $$\C=\C(I_1,\dots,I_n): \quad 0\to C_0\to C_1 \to C_2 \to \dots \to  C_n \to 0$$ 
 and set $H^i=H^i(I_1, \dots, I_n)=H^i(\C)$ so that  $I_{\VV}=H^0$. 

We know by Proposition \ref{smalldim} that $\dim H^i=0$ for all $i>0$. 

\begin{remark} 
By assumption we have $\dim C_i=\max\{ r-2i, 0\}$ and $C_i$ is the  direct sum of $\binom{n}{i}$ polynomial rings of dimension $\dim C_i$ whose local cohomology modules are hence well understood. In particular, $\reg(C_i)=0$ for all $i$.  
\end{remark} 

The discussion will be divided in two cases. The first and easy case is the following: 

\begin{theorem}
\label{2n leq r}
Assume $r\geq 2n$. Then $H^{r-i}_\mm(I_{\VV})=H^{r-2i}_\mm(C_i)$ for $i=0,\dots, n$ and $H^{r-i}_\mm(I_{\VV})=0$ for $i>n$. 
\end{theorem} 
\begin{proof} 
In this case $H^i=0$ for all $i>0$ by Theorem \ref{thm:monomial} because, after a suitable change of coordinates, $I_1,\dots, I_n$ are generated by variables. Arguing by induction on $n$ and splitting the complex $\C(I_1,\dots,I_n)$ into short exact sequences we get the desired result. 
\end{proof} 

From now on we assume that $r<2n$. Then we set 
 $s=\max\{ i : r-2i>0\}$,  that is, 
 $$ s=\lfloor  (r-1)/2 \rfloor$$
 and truncate the complex $\C(I_1,\dots,I_n)$  on the right at position $s$: 
$$\hat\C(I_1,\dots,I_n):\quad  0\to C_0\to C_1 \to C_2 \to \dots \to  C_s \to 0.$$ 
Denote by $\hat H^i$ the $i$-th cohomology of $\hat\C(I_1,\dots,I_n)$. 
\begin{proposition}
\label{cohocompa} 
For all $i<s$ we have $\hat H^i=H^i$. For $i=s$,  $\left[\hat H^s \right]_{j \neq 0} = \left[H^s \right]_{j \neq 0}$ and 
$$\dim_K \left[\hat H^s\right]_0 =\binom{n-1}{s}.$$
\end{proposition} 
\begin{proof} 
The first assertion is obvious. The second assertion follows from the fact that $C_{s+1}$ is concentrated in degree $0$. Finally observe that the complex $\C(I_1,\dots,I_n)$ in degree $0$ coincides with  $(\wedge K^n, L)$ and hence has no cohomology. Its truncation at position $s$ then acquires cohomology of dimension 
$$\sum_{i=1}^{n-s} (-1)^{i+1}  \binom{n}{s+i} = \binom{n-1}{s}. \qedhere$$
\end{proof} 

Now we observe that the complex $\hat\C(I_1,\dots,I_n)$ satisfies the assumptions of \ref{sepcoho} and that we are in the even case for $r$ even and in the odd case when $r$ is odd. Therefore Theorems \ref{evencase}  and \ref{oddcase} apply and we can express the local cohomology modules of $I_V$ in terms of the cohomology of $\hat\C(I_1,\dots,I_n)$ and that of $C_i$ with $i=0,\dots, s$. Taking into account Proposition \ref{cohocompa} and that the local cohomology modules of $C_i$ live only in the negative degrees, we may rephrase the final results as follows. 

\begin{theorem}
\label{cohomology-even}
Assume $r$ is even. Then 
$$H_\mm^{i}(I_\VV)=\left\{ 
\begin{array}{ll}
H^{i-1}(\C) &  \mbox{ for } i=2,\dots, s \\ 
H_\mm^{2i-2s-2}(C_{2s-i+2}) & \mbox{ for } i=s+2,\dots, 2s+2=r\\
 0 &  \mbox{  for  }  i<2 \mbox{ or } i>2s+2 
 \end{array}
 \right.
 $$
 and furthermore
 $$\left[H_\mm^{s+1}(I_\VV)\right]_{j>0}=\left[H^s(\C)\right]_{j>0}, \quad \left[H_\mm^{s+1}(I_\VV)\right]_{j<0}=0, $$ while 
 $$\dim_K \left[H_\mm^{s+1}(I_\VV)\right]_0=\binom{n-1}{s}.$$
 \end{theorem}

\begin{theorem}
\label{cohomology-odd}
Assume $r$ is odd. Then 

$$H_\mm^{i}(I_\VV)=\left\{ 
\begin{array}{ll}
H^{i-1}(\C) &  \mbox{ for } i=2,\dots, s \\ 
H_\mm^{2i-2s-1}(C_{2s-i+1}) & \mbox{ for } i=s+2,\dots, 2s+1=r\\
 0 &  \mbox{  for  }  i<2 \mbox{ or } i>2s+2 
 \end{array}
 \right.
 $$
 and furthermore 
 $$\left[H_\mm^{s+1}(I_\VV)\right]_{j>0}=\left[H^s(\C)\right]_{j>0}, \quad \left[H_\mm^{s+1}(I_\VV)\right]_{j<0}=\left[H^1_\mm(C_s)\right]_{j<0}, $$ while 
 $$\dim_K \left[H_\mm^{s+1}(I_\VV)\right]_0=\binom{n-1}{s}.$$
 \end{theorem} 

The Castelnuovo-Mumford regularity of the vanishing ideal of a subspace arrangement in projective space has been studied in generality in \cite{DS}, and in the transversal case in \cite{CT} and more recently in \cite{HHT}. Here, as a consequence of the above, we obtain a new formula for a codimension-$2$ transversal subspace arrangement in $\mathbb{P}^{r-1}$.

\begin{corollary}
\label{regIV}
Under Assumption \ref{ass:LinearlyGeneralDim=2}, we have
 $$\reg(I_\VV)= \max
 \left\{ 
 \begin{array}{lll}
  \reg(H^i(\C))+i+1 &:&  i=1,\dots, s, \\
  s+1
 \end{array} 
 \right.$$
\end{corollary}


\begin{thebibliography}{99}

\bibitem{BCH}
A.~Banerjee, M.~Chardin, and R.~Holanda, 
{\em Multiple Tor Modules: Rigidity and Mayer-Vietoris Spectral Sequences}, 
arXiv:2604.20646 [math.AC], 2026. 

\bibitem{CT} 
A. Conca and M. C. Tsakiris, 
{\em On the Castelnuovo--Mumford Regularity of Subspace Arrangements,} 
Journal of the European Mathematical Society (2026).  

\bibitem{Derksen}
H.~Derksen, 
{\em Hilbert Series of Subspace Arrangements,} 
Journal of Pure and Applied Algebra 209 (2007), no. 1, 91--98.  

\bibitem{DS}
H.~Derksen, J.~Sidman
{\em A Sharp Bound for the Castelnuovo-Mumford Regularity of Subspace Arrangements,} 
Advances in Mathematics 172 (2002), 151--157.  

\bibitem{HHT} 
X. Hong, Y.-S. Huang, and M. C. Tsakiris, 
{\em A Sublinear Bound for the Regularity of Subspace Arrangements,}
 arXiv:2607.18892 [math.AC, math.AG], 2026.  

\bibitem{M2} D. R. Grayson and M. E. Stillman, 
 {\em Macaulay2, A Software System for Research in Algebraic Geometry,}
  available at https://macaulay2.com/.  

\end{thebibliography}
\end{document}